\documentclass{amsart}

\usepackage{tikz}
\usetikzlibrary{calc}
\usepackage{xspace}

\theoremstyle{plain}
\newtheorem{prop}{Proposition}[section]
\newtheorem{thrm}[prop]{Theorem}
\newtheorem{prob}[prop]{Problem}

\newtheorem{lem}[prop]{Lemma}
\theoremstyle{definition}
\newtheorem{defi}[prop]{Definition}
\newtheorem{exam}[prop]{Example}
\newtheorem*{nota}{Notation}

\newcommand\ie{\emph{i.e.}}
\newcommand\inv{^{-1}}
\newcommand\kbr[1]{\left<#1\right>}
\newcommand\NN{\mathbb{N}}
\newcommand\ZZ{\mathbb{Z}}
\newcommand\ti{\tikz[baseline,x=18,y=18]{\begin{scope}[shift={(0,0.18)}]
\draw[line width=1.2](-0.2,-0.2) .. controls (-0.1,0) .. (-0.2,0.2);
\draw[line width=1.2](0.2,-0.2) .. controls (0.1,0) .. (0.2,0.2);
\end{scope}}\xspace}

\newcommand\toneinv{\tikz[baseline,x=18,y=18]{\begin{scope}[shift={(0,0.18)}]
\draw[line width=1.2](-0.2,0.2) -- (0.2,-0.2);
\draw[line width=3.6,color=white](-0.2,-0.2) -- (0.2,0.2);
\draw[line width=1.2](-0.2,-0.2) -- (0.2,0.2);
\end{scope}}\xspace}
\newcommand\tz{\tikz[baseline,x=18,y=18]{\begin{scope}[shift={(0,0.18)}]
\draw[line width=1.2](-0.2,-0.2) .. controls (0,-0.1) .. (0.2,-0.2);
\draw[line width=1.2](-0.2,0.2) .. controls (0,0.1) .. (0.2,0.2);
\end{scope}}\xspace}
\newcommand\unknot{\tikz[baseline,x=18,y=18]{\begin{scope}[shift={(0,0.18)}]
\draw [ line width=1.2] (0,0) circle(0.2);\end{scope}}}
\DeclareMathOperator\br{br}
\DeclareMathOperator\den{den}
\DeclareMathOperator\num{num}
\DeclareMathOperator\lt{lt}
\DeclareMathOperator\wri{wr}

\title{A remarkable 20-crossing tangle}
\author{Shalom Eliahou}
\author{Jean Fromentin}


\begin{document}

\begin{abstract}
For any positive integer $r$, we exhibit a nontrivial knot $K_r$ with $(20\cdot 
2^{r-1}+1)$ crossings whose Jones polynomial $V(K_r)$ is equal to $1$ 
modulo~$2^r$. Our construction rests on a certain $20$-crossing tangle 
$T_{20}$ which is undetectable by the Kauffman bracket polynomial pair mod $2$. 
\end{abstract}	

\maketitle

\section{Introduction}

In \cite{T},  M.\,B.\,Thistlethwaite gave two $2$--component links
and one $3$--component link which are nontrivial and yet have the same Jones polynomial 
as the corresponding unlink~$U^2$ and $U^3$, respectively.
These were the first known examples of nontrivial links undetectable 
by the Jones polynomial.
Shortly thereafter, it was shown in \cite{EKP} that, for any integer $k\geq 2$, 
there exist infinitely many nontrivial $k$--component links whose Jones polynomial
is equal to that of the $k$--component unlink~$U^k$. 
Yet the corresponding problem for $1$--component links, i.e. for knots, is widely open: 
\emph{does there exist a nontrivial knot $K$ whose Jones polynomial 
is equal to that of the unknot $U^1$, namely to 1?}

We shall consider here the following weaker problem, consisting in looking for nontrivial knots $K$ whose Jones polynomial is \emph{congruent modulo some integer} to that of the unknot $U^1$.

\begin{prob}
\label{Prob1}
Given any integer $m\geq 2$, does there exist a nontrivial knot $K$ whose 
Jones polynomial $V(K)$ satisfies $V(K)\equiv 1 \mod m$?
\end{prob}

Naturally, for any two Laurent polynomials $f$, $g$ in $\ZZ[t,t\inv]$, 
the notation $f\equiv g\mod m$ means that there exists an element
$h\in\ZZ[t,t\inv]$ such that $f-g=m\cdot h$. 
This is equivalent to require that, for each $i\in\ZZ$, 
the coefficients $\alpha_i$ and $\beta_i$ of $t^i$ in $f$ and $g$, 
respectively, are congruent modulo $m$ as integers.

A result of M.\,B.\,Thistlethwaite \cite{MR899051} states that, for an alternating knot $K$ with $n$ crossings, 
the span of $V(K)$ is exactly $n$ and the coefficients of the terms of maximal and minimal degree in $V(K)$ are both $\pm1$. 
In particular, for any $m\geq 2$, \emph{there is no alternating knot with trivial Jones polynomial modulo~$m$}. 

Using the \emph{Mathematica} package \texttt{KnotTheory} of the KnotAtlas 
project \cite{knotatlas}, it is easy to find  knots which are 
solutions of Problem~\ref{Prob1} for the moduli $m=2,3$ and~$4$. 
For $m=5$ there is no solution of Problem~\ref{Prob1}
among the knots up to $16$ crossings. 
The following table gives the number of solutions of Problem~\ref{Prob1} 
for the moduli $2$ to $5$ up to $16$ crossings, respectively: 
$$
\begin{array}{|c||c|c|c|c|c|c|}
\hline
m &\leq11 & 12 & 13 & 14 & 15 & 16\\
\hline
2 & 0 & 4 & 9 & 35 & 140 & 582 \\
3 & 0 & 1 & 0 & 1 & 2 & 26  \\
4 & 0 & 0 & 0 & 0 & 1 & 0 \\
5 & 0 & 0 & 0 & 0 & 0 & 0\\
\hline
\end{array}
$$

In this note, we shall solve Problem~\ref{Prob1} for all moduli $m$ which are 
powers of~$2$. 
That is, given any integer $r\geq 1$, we shall construct infinitely many 
knots whose Jones polynomial is trivial mod $m=2^r$. 
Our construction rests on a certain $20$-crossing tangle $T_{20}$ whose Kauffman
bracket polynomial pair is trivial mod $2$.

The paper is structured as follows.
Section~\ref{S:basic} is devoted to basic tangle operations.
In Section~\ref{S:lickorish} we describe our tangle $T_{20}$ and we construct a family of knots~$K_r$ from a tangle 
$M_r$ which is composed of~$2^r$ copies of the tangle $T_{20}$.
In Section~\ref{S:bracket}, we compute the Kauffman bracket pair of the tangle $M_r$ modulo $2^r$.
In Section~\ref{S:jones}, we prove that the knots~$K_r$ for $r\geq 1$ are 
distinct and that the Jones polynomial of~$K_r$ is trivial modulo $2^r$. 
The paper concludes with a few related open questions.

\begin{center}
  \begin{figure}[h!]
    \begin{tikzpicture}[x=0.05cm,y=0.05cm]
	\draw(-5,10) node{\textbf{a.}};
	\coordinate (O1) at (6,10);
	\draw (O1) circle (6);
	\draw (O1) node{\large $T_1$};
	\coordinate (O2) at (26,10);
	\draw (O2) circle (6);
	\draw (O2) node{\large $T_2$};
	\draw[line width=2] (O1) ++ (45:6) -- ++ (45:4) coordinate (NE);
	\draw[line width=2] (O2) ++ (135:6) -- ++ (135:4) coordinate (NW);
	\draw[line width=2] (O1) ++ (315:6) -- ++ (315:4) coordinate (SE);
	\draw[line width=2] (O2) ++ (225:6) -- ++ (225:4) coordinate (SW);
	\draw[line width=2] (O1) ++ (135:6) -- ++ (135:4);
	\draw[line width=2] (O1) ++ (225:6) -- ++ (225:4);
	\draw[line width=2] (O2) ++ (45:6) -- ++ (45:4);
	\draw[line width=2] (O2) ++ (315:6) -- ++ (315:4);
	\coordinate (C1) at (intersection of O1--NE and O2--NW);
	\draw[line width=2] (NE) .. controls ($(NE)!0.3!(C1)$) and 
($(C1)!0.3!(NW)$) .. (NW);
	\coordinate (C2) at (intersection of O1--SE and O2--SW);
	\draw[line width=2] (SE) .. controls ($(SE)!0.3!(C2)$) and 
($(C2)!0.3!(SW)$) .. (SW);

      \begin{scope}[shift={(55,-10)}]
	\draw(-7,20) node{\textbf{b.}};
	\coordinate (O1) at (6,10);
	\draw (O1) circle (6);
	\draw (O1) node{\large $T_2$};
	\coordinate (O2) at (6,30);
	\draw (O2) circle (6);
	\draw (O2) node{\large $T_1$};
	\draw[line width=2] (O1) ++ (45:6) -- ++ (45:4) coordinate (NE);
	\draw[line width=2] (O1) ++ (135:6) -- ++ (135:4) coordinate (NW);
	\draw[line width=2] (O2) ++ (225:6) -- ++ (225:4) coordinate (SW);
	\draw[line width=2] (O2) ++ (315:6) -- ++ (315:4) coordinate (SE);
	\draw[line width=2] (O2) ++ (45:6) -- ++ (45:4);
	\draw[line width=2] (O2) ++ (135:6) -- ++ (135:4);
	\draw[line width=2] (O1) ++ (225:6) -- ++ (225:4);
	\draw[line width=2] (O1) ++ (315:6) -- ++ (315:4);
	\coordinate (C1) at (intersection of O1--NW and O2--SW);
	\draw[line width=2] (NW) .. controls ($(NW)!0.3!(C1)$) and 
($(C1)!0.3!(SW)$) .. (SW);
	\coordinate (C2) at (intersection of O1--NE and O2--SE);
	\draw[line width=2] (NE) .. controls ($(NE)!0.3!(C2)$) and 
($(C2)!0.3!(SE)$) .. (SE);
      \end{scope}
      \begin{scope}[shift={(90,0)}]
	\draw(-10,0) node{\textbf{d.}};
	\draw[line width=2] (-6,6) .. controls (0,6) and (0,-6) .. (6,-6);
	\draw[line width=6,color =white] (-6,-6) .. controls (0,-6) and (0,6) 
.. (6,6);
	\draw[line width=2](-6,-6) .. controls (0,-6) and (0,6) .. (6,6);
      \end{scope}
      \begin{scope}[shift={(90,20)}]
	\draw(-10,0) node{\textbf{c.}};
	\draw[line width=2](-6,-6) .. controls (0,-6) and (0,6) .. (6,6);
	\draw[line width=6,color =white] (-6,6) .. controls (0,6) and (0,-6) .. 
(6,-6);
	\draw[line width=2] (-6,6) .. controls (0,6) and (0,-6) .. (6,-6);
      \end{scope}
      \begin{scope}[shift={(125,9)}]
        \draw(-17,0) node{\textbf{e.}};
        \draw(0,0) coordinate (O) circle (6);
        \draw(O) node{$T$};	
        \draw[line width=2] (45:6) .. controls (45:20) and (315:20) .. (315:6);
        \draw[line width=2] (135:6) .. controls (135:20) and (225:20) .. 
(225:6);
      \end{scope}
      \begin{scope}[shift={(160,9)}]
        \draw(-12,0) node{\textbf{f.}};
        \draw(0,0) coordinate (O) circle (6);
        \draw(O) node{$T$};	
        \draw[line width=2] (135:6) .. controls (135:20) and (45:20) .. (45:6);
        \draw[line width=2] (225:6) .. controls (225:20) and (315:20) .. 
(315:6);
      \end{scope}
    \end{tikzpicture}
    \caption{
      \textbf{a.}~Tangle  $T_1+T_2$. 
      \textbf{b.}~Tangle $T_1\ast T_2$.
      \textbf{c.}~Tangle $1$.
      \textbf{d.}~Tangle $-1$.
      \textbf{e.}~Closure $\den(T)$.
      \textbf{f.}~Closure $\num(T)$.
    }
    \label{F:Tangle}
  \end{figure}
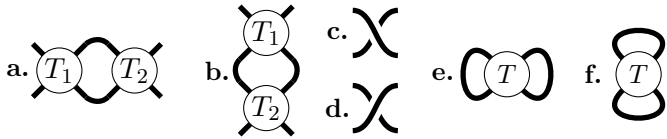
\end{center}

\section{Basic tangle operations}
\label{S:basic}

We shall use the same notation as in  \cite{EKP}. In particular, if $T_1$, $T_2$ are two tangles with $4$ endpoints, 
we denote by $T_1+T_2$ their \emph{horizontal sum} and by $T_1\ast T_2$ their \emph{vertical 
sum} (see \textbf{a} and \textbf{b} of Figure~\ref{F:Tangle}).
The tangle $1$ denotes a single crossing as in \textbf{c} of 
Figure~\ref{F:Tangle}, while the tangle~$-1$ denotes its opposite version as in 
\textbf{d} of Figure~\ref{F:Tangle}.
More generally, if $T$ is a tangle, then $-T$ denotes the tangle obtained from $T$ 
by switching the signs of all crossings in $T$.

As usual, for $k\in \NN\setminus\{0\}$ we define the tangles
\begin{align*}
k&=1+\cdots+1, &-k=(-1)+\cdots+(-1),\\
 1/k&=1\ast\cdots\ast1, &-1/k=(-1)\ast\cdots\ast(-1), 
 \end{align*}
with $k$ terms in each expression. These are particular cases of \emph{algebraic tangles}, namely tangles constructed recursively from the tangles $\pm 1$ using horizontal and vertical sums.

\begin{exam}
\label{E:Tangle1}
Let $T_{8,21}$ be the tangle $(((1/2)+1)\ast 2)+(-3)$. (See next example for the choice of this name.)
Diagrams corresponding to $T_{8,21}$ are
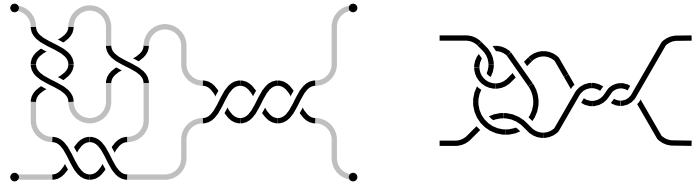
\begin{center}
\begin{tikzpicture}[x=0.05cm,y=0.05cm]
\draw[line width=2](0,30) .. controls (0,35) and (10,35) .. (10,40);
\draw[line width=6,color=white](0,40) .. controls (0,35) and (10,35) .. (10,30);
\draw[line width=2](0,40) .. controls (0,35) and (10,35) .. (10,30);
\draw[line width=2](0,20) .. controls (0,25) and (10,25) .. (10,30);
\draw[line width=6,color=white](0,30) .. controls (0,25) and (10,25) .. (10,20);
\draw[line width=2](0,30) .. controls (0,25) and (10,25) .. (10,20);
%1
\draw[line width=2](20,25) ..controls (20,30) and (30,30) .. (30,35);
\draw[line width=6,color=white](20,35) ..controls (20,30) and (30,30) .. 
(30,25);
\draw[line width=2](20,35) ..controls (20,30) and (30,30) .. (30,25);
%2
\draw[line width=2](5,0) .. controls (10,0) and (10,10) .. (15,10);
\draw[line width=6,color=white](5,10) .. controls (10,10) and (10,0) .. (15,0);
\draw[line width=2](5,10) .. controls (10,10) and (10,0) .. (15,0);
\draw[line width=2](15,0) .. controls (20,0) and (20,10) .. (25,10);
\draw[line width=6,color=white](15,10) .. controls (20,10) and (20,0) .. (25,0);
\draw[line width=2](15,10) .. controls (20,10) and (20,0) .. (25,0);
%-3
\draw[line width=2](45,25) .. controls (50,25) and (50,15) .. (55,15);
\draw[line width=6,color=white](45,15) .. controls (50,15) and (50,25) .. 
(55,25);
\draw[line width=2](45,15) .. controls (50,15) and (50,25) .. (55,25);
\draw[line width=2](55,25) .. controls (60,25) and (60,15) .. (65,15);
\draw[line width=6,color=white](55,15) .. controls (60,15) and (60,25) .. 
(65,25);
\draw[line width=2](55,15) .. controls (60,15) and (60,25) .. (65,25);
\draw[line width=2](65,25) .. controls (70,25) and (70,15) .. (75,15);
\draw[line width=6,color=white](65,15) .. controls (70,15) and (70,25) .. 
(75,25);
\draw[line width=2](65,15) .. controls (70,15) and (70,25) .. (75,25);
%Connections
%1/2 -- 1
\draw[line width=2,color=lightgray](10,40) arc(180:0:5) -- (20,35);
\draw[line width=2,color=lightgray](10,20) arc(-180:0:5) -- (20,25);
%1/2 -- 2
\draw[line width=2,color=lightgray](0,20) -- (0,15) arc (180:270:5);
%2 -- 1
\draw[line width=2,color=lightgray] (25,10) arc (-90:0:5) -- (30,25);
%2 -- -3
\draw[line width=2,color=lightgray] (25,0) -- (35,0)  arc (-90:0:5) -- (40,10) 
arc(180:90:5);
%1 -- -3
and (40,25) .. (45,25);
\draw[line width=2,color=lightgray] (30,35) arc (180:0:5) -- (40,30) arc 
(180:270:5);
%Extremities
\draw[line width=2,color=lightgray](-5,0) -- (5,0);
\draw[line width=2,color=lightgray](0,40) arc (0:90:5);
\draw[line width=2,color=lightgray](75,25) arc (-90:0:5) -- (80,40) arc 
(180:90:5);
\draw[line width=2,color=lightgray](75,15) arc (90:0:5) -- (80,5) arc 
(180:270:5);
\filldraw[color=black](-5,0) circle(1);
\filldraw[color=black](-5,45) circle(1);
\filldraw[color=black](85,0) circle(1);
\filldraw[color=black](85,45) circle(1);
\begin{scope}[shift={(110,-3)}]
\draw[line width=2](-2,40) -- (5,40) arc(90:45:5) 
coordinate (A0);
\draw[line width=2] (A0) -- ++ (-45:2) arc (45:0:6) coordinate (A1);
\draw[line width=2] (A1) arc(0:-45:6) arc (135:180:8) coordinate (A2);
\draw[line width=2] (A2) arc(180:270:8) coordinate(A3);
\draw[line width=2] (A3) arc(270:395:8) coordinate (A4);
\draw[line width=2] (A4)-- ++ (125:10) coordinate (A5);
\draw[line width=2] (A5) arc(45:180:5) coordinate (A6);
\draw[line width=6,color=white] (A0) -- ++ (-45:2) arc (45:0:6);
\draw[line width=2] (A0) -- ++ (-45:2) arc (45:0:6);
\draw[line width=6,color=white] (A6) arc(180:270:5) coordinate(A7);
\draw[line width=2] (A6) arc(180:270:5) coordinate(A7);
\draw[line width=2] (A7) arc(270:315:5) --++ (45:8) coordinate(A8);
\draw[line width=6,color=white] (A4)-- ++ (125:10);
\draw[line width=2] (A4)-- ++ (125:10);
\draw[line width=2] (A8) arc (135:45:5) -- ++ (-60:5) coordinate(A9);
\draw[line width=2] (A9) -- ++ (-60:6) arc(-150:-90:4.5) coordinate (A10);
\draw[line width=2] (A10) arc(-90:-35:4.5)-- ++ (55:2) 
arc(-225:-270:4) coordinate (A11);
\draw[line width=2] (A11)  arc(-270:-330:4) -- ++ (-60:13.5) arc (224:270:5) -- 
(65,12) ;

\draw[line width=2](-2,12) -- (2.2,12) arc (270:315:5) -- ++(45:5) 
arc(135:90:8) coordinate(B1);
\draw[line width=6,color=white] (A2) arc(180:270:8);
\draw[line width=2] (A2) arc(180:270:8);
\draw[line width=6,color=white] (B1) arc(90:45:8) -- ++ (-45:2) arc (225:270:5) 
coordinate(B2);
\draw[line width=2] (B1) arc(90:45:8) -- ++ (-45:2) arc (225:270:5) 
coordinate(B2);
\draw[line width=6,color=white] (B2) arc(270:315:5) -- 
++ (60:11) arc (150:90:4.5);
\draw[line width=2] (B2) arc(270:315:5) -- 
++ (60:11) arc (150:90:4.5) coordinate(B3);
\draw[line width=2] (B3) arc(90:35:4.5) -- ++ (-55:2) arc 
(225:270:4) coordinate (B4);
\draw[line width=6,color=white] (B4) arc(270:330:4) -- ++ (60:16) arc 
(135:90:5) -- (65,40);
\draw[line width=2] (B4) arc(270:330:4) -- ++ (60:16) arc (135:90:5) -- (65,40);
\draw[line width=6,color=white] (A10) arc(-90:-35:4.5)-- ++ (55:2) 
arc(-225:-270:4) coordinate (A11);
\draw[line width=2] (A10) arc(-90:-35:4.5)-- ++ (55:2) 
arc(-225:-270:4) coordinate (A11);
\end{scope}
\end{tikzpicture}

\end{center}
The parts in black correspond to the tangles $1/2$, $1$, $2$ and $-3$, respectively, while the gray 
strands depict the connections between them.
The four marked points represent the extremities of the global tangle.
The rightmost diagram is a ``smoother'' representation of the tangle~$T_{8,21}$.
\end{exam}

If $T$ is a tangle, we denote by $\den(T)$ and $\num(T)$ the link diagrams obtained by 
gluing the extremities of $T$ as in \textbf{e} and \textbf{f} of 
Figure~\ref{F:Tangle}, respectively.

\begin{exam}
\label{E:Tangle2}
Consider again the tangle $T_{8,21}$ of Example~\ref{E:Tangle1}. Here are 
$\num(T_{8,21})$ and an isotopic diagram.
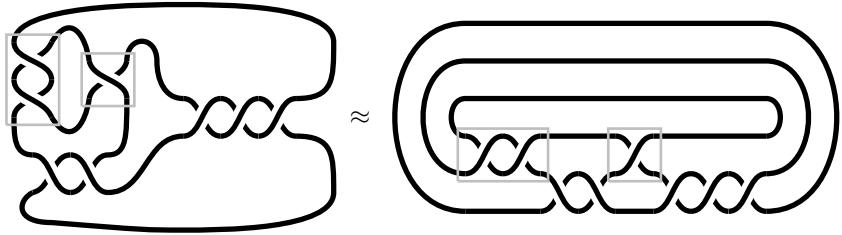
\begin{center}
\begin{tikzpicture}[x=0.05cm,y=0.05cm]
%1/2
\draw[line width=2](0,30) .. controls (0,35) and (8,35) .. (10,40);
\draw[line width=6,color=white](0,40) .. controls (0,35) and (10,35) .. (10,30);
\draw[line width=2](0,40) .. controls (0,35) and (10,35) .. (10,30);
\draw[line width=2](0,20) .. controls (0,25) and (10,25) .. (10,30);
\draw[line width=6,color=white](0,30) .. controls (0,25) and (8,25) .. (10,20);
\draw[line width=2](0,30) .. controls (0,25) and (8,25) .. (10,20);
%1
\draw[line width=2](20,25) ..controls (21,30) and (29,30) .. (30,35);
\draw[line width=6,color=white](20,35) ..controls (21,30) and (30,30) .. 
(30,25);
\draw[line width=2](20,35) ..controls (21,30) and (30,30) .. (30,25);
%2
\draw[line width=2](5,0) .. controls (10,1) and (10,10) .. (15,10);
\draw[line width=6,color=white](5,10) .. controls (10,10) and (10,0) .. (15,0);
\draw[line width=2](5,10) .. controls (10,10) and (10,0) .. (15,0);
\draw[line width=2](15,0) .. controls (20,0) and (20,10) .. (25,10);
\draw[line width=6,color=white](15,10) .. controls (20,10) and (20,0) .. (25,0);
\draw[line width=2](15,10) .. controls (20,10) and (20,0) .. (25,0);
%-3
\draw[line width=2](45,25) .. controls (50,25) and (50,15) .. (55,15);
\draw[line width=6,color=white](45,15) .. controls (50,15) and (50,25) .. 
(55,25);
\draw[line width=2](45,15) .. controls (50,15) and (50,25) .. (55,25);
\draw[line width=2](55,25) .. controls (60,25) and (60,15) .. (65,15);
\draw[line width=6,color=white](55,15) .. controls (60,15) and (60,25) .. 
(65,25);
\draw[line width=2](55,15) .. controls (60,15) and (60,25) .. (65,25);
\draw[line width=2](65,25) .. controls (70,25) and (70,15) .. (75,15);
\draw[line width=6,color=white](65,15) .. controls (70,15) and (70,25) .. 
(75,25);
\draw[line width=2](65,15) .. controls (70,15) and (70,25) .. (75,25);
%Connections
%1/2 -- 1
\draw[line width=2](10,40) .. controls (15,48) and (19,42) .. (20,35);
\draw[line width=2](10,20) .. controls (15,12) and (19,18) .. (20,25);
%1/2 -- 2
\draw[line width=2] (0,20) .. controls (0,15) and (0,10) .. (5,10);
\draw[line width=2] (25,10) .. controls (30,10) and (30,15) .. (30,25);
%2 -- -3
\draw[line width=2] (25,0) .. controls (35,0) and (35,15) .. (45,15);
%1/2 -- -3
\draw[line width=2] (30,35) .. controls (30,42) and (38,42) .. (38,35) .. 
controls (38,30) and (40,25) .. (45,25);
%Closure
\draw[line width=2] (0,40) .. controls (0,50) and (35,50) .. (45,50) .. 
controls (55,50) and (85,50) .. (85,40) .. controls (85,30) and (85,25) .. 
(75,25);
\draw[line width=2](5,0) .. controls (1,-2)  and (0,-8) .. (10,-8) .. controls 
(35,-10) ..(45,-10) ..controls (55,-10) and (85,-10) .. (85,0) .. controls 
(85,10) and (85,15) .. (75,15); 
%Blocks
\draw[line width=1,color=lightgray](-2,18) -- (12,18) -- (12,42) -- (-2,42) -- 
(-2,18) ;
\draw[line width=1,color=lightgray](18,23) -- (32,23) -- (32,37) -- (18,37) -- 
(18,23);
\draw(92,20) node{$\approx$};
%braid view 
\begin{scope}[shift={(115,-10)}]
\draw[line width=2](5,25) .. controls (10,25) and (10,15) .. (15,15);
\draw[line width=6,color=white](5,15) .. controls (10,15) and (10,25) .. 
(15,25);
\draw[line width=2](5,15) .. controls (10,15) and (10,25) .. (15,25);
\draw[line width=2](15,25) .. controls (20,25) and (20,15) .. (25,15);
\draw[line width=6,color=white](15,15) .. controls (20,15) and (20,25) .. 
(25,25);
\draw[line width=2](15,15) .. controls (20,15) and (20,25) .. (25,25);
\draw[line width=2](25,25) -- (45,25);
\draw[line width=2](45,25) .. controls (50,25) and (50,15) .. (55,15);
\draw[line width=6,color=white](45,15) .. controls (50,15) and (50,25) .. 
(55,25);
\draw[line width=2](45,15) .. controls (50,15) and (50,25) .. (55,25);

\draw[line width=2](25,5) .. controls (30,5) and (30,15) .. (35,15);
\draw[line width=6,color=white](25,15) .. controls (30,15) and (30,5) .. (35,5);
\draw[line width=2](25,15) .. controls (30,15) and (30,5) .. (35,5);
\draw[line width=2](35,5) .. controls (40,5) and (40,15) .. (45,15);
\draw[line width=6,color=white](35,15) .. controls (40,15) and (40,5) .. (45,5);
\draw[line width=2](35,15) .. controls (40,15) and (40,5) .. (45,5);

\draw[line width=2](55,15) .. controls (60,15) and (60,5) .. (65,5);
\draw[line width=6,color=white](55,5) .. controls (60,5) and (60,15) .. (65,15);
\draw[line width=2](55,5) .. controls (60,5) and (60,15) .. (65,15);
\draw[line width=2](65,15) .. controls (70,15) and (70,5) .. (75,5);
\draw[line width=6,color=white](65,5) .. controls (70,5) and (70,15) .. (75,15);
\draw[line width=2](65,5) .. controls (70,5) and (70,15) .. (75,15);
\draw[line width=2](75,15) .. controls (80,15) and (80,5) .. (85,5);
\draw[line width=6,color=white](75,5) .. controls (80,5) and (80,15) .. (85,15);
\draw[line width=2](75,5) .. controls (80,5) and (80,15) .. (85,15);

\draw[line width=2](45,5) -- (55,5);
\draw[line width=2](5,5) -- (25,5);
\draw[line width=2](55,25) -- (85,25);
\draw[line width=2](5,25) .. controls (0,25) and (0,35) .. (5,35) -- (85,35) .. 
controls (90,35) and (90,25) .. (85,25);
\draw[line width=2](5,15) .. controls (-10,15) and (-10,45) .. (5,45) -- 
(85,45) .. controls (100,45) and (100,15) .. (85,15);
\draw[line width=2](5,5) .. controls (-20,5) and (-20,55) .. (5,55) -- (85,55) 
.. controls (110,55) and (110,5) .. (85,5);
%Block
\draw[line width=1,color=lightgray](3,13) -- (27,13) -- (27,27) -- (3,27) -- 
(3,13); 
\draw[line width=1,color=lightgray](43,13) -- (57,13) -- (57,27) -- (43,27) -- 
(43,13); 

\end{scope}
 \end{tikzpicture}
 \end{center}
The isotopy is obtained from the left diagram by rotating counterclockwise 
the left block in gray and clockwise the right one. 
Looking at Dale Rolfen's knot table~\cite{Rolfsen}, we remark that 
$\num(T_{8,21})$ is a diagram of $K_{8,21}$ which is the $21$\textsuperscript{st} prime knot with $8$ crossings.
\end{exam}

\medskip
\section{The main construction}
\label{S:lickorish}

% For a more geometric point of view on tangles, we need the following refinement as in \cite{EKP}.
% A \emph{geometric tangle} is a pair $(B,t)$, where $B$ is a $3$-ball and 
% $t$ is a proper $1$-submanifold of $B$ meeting the boundary of $B$ in four 
% points.
% 
% To each tangle $T$ we associate in a natural way a geometric tangle $(B,t)$ where~$B$ is a Euclidean $3$-ball 
% whose boundary meets the projection plane $P$ in an ``equatorial'' circle circumscribing the tangle $T$, and where $t$
% is obtained from~$T$ drawn in~$P$ by making small vertical perturbations near crossings.

%\subsection{The tangles $M_r$}
We now introduce a family of tangles which includes the tangle $T_{20}$, the cornerstone
of our solution to Problem~\ref{Prob1} for moduli $m$ which are powers of $2$.

\begin{defi}\label{D:Tangles}
We define the tangles $T_{10}$, $T_{20}$ and $M_r$ for $r \ge 1$ as follows: 

-- $i)$ $T_{10}=T_{8,21}\ast 2=((((1/2)+1)\ast2)+(-3))\ast2$;

-- $ii)$ $T_{20}=T_{10}+(-T_{10})$;

-- $iii)$ $M_1=T_{20}$ and 
$M_r=M_{r-1}+M_{r-1}$ for $r\geq 2$.
 \end{defi}

 Tangle $M_1$ is depicted in Figure~\ref{F:M1}.
 
 \begin{figure}[h!]
 \begin{center}
\begin{tikzpicture}[x=0.05cm,y=0.05cm]
\draw[line width=2](-2,40) -- (5,40) arc(90:45:5) 
coordinate (A0);
\draw[line width=2] (A0) -- ++ (-45:2) arc (45:0:6) coordinate (A1);
\draw[line width=2] (A1) arc(0:-45:6) arc (135:180:8) coordinate (A2);
\draw[line width=2] (A2) arc(180:270:8) coordinate(A3);
\draw[line width=2] (A3) arc(270:395:8) coordinate (A4);
\draw[line width=2] (A4)-- ++ (125:10) coordinate (A5);
\draw[line width=2] (A5) arc(45:180:5) coordinate (A6);
\draw[line width=6,color=white] (A0) -- ++ (-45:2) arc (45:0:6);
\draw[line width=2] (A0) -- ++ (-45:2) arc (45:0:6);
\draw[line width=6,color=white] (A6) arc(180:270:5) coordinate(A7);
\draw[line width=2] (A6) arc(180:270:5) coordinate(A7);
\draw[line width=2] (A7) arc(270:315:5) --++ (45:8) coordinate(A8);
\draw[line width=6,color=white] (A4)-- ++ (125:10);
\draw[line width=2] (A4)-- ++ (125:10);
\draw[line width=2] (A8) arc (135:45:5) -- ++ (-60:5) coordinate(A9);
\draw[line width=2] (A9) -- ++ (-60:6) arc(-150:-90:4.5) coordinate (A10);
\draw[line width=2] (A10) arc(-90:-35:4.5)-- ++ (55:2) arc(-225:-270:4) 
coordinate (A11);
\draw[line width=2] (A11)  arc(90:-90:10) coordinate (C1);
\draw (C1) arc (90:135:8) -- ++ (225:2) arc (-45:-90:8) coordinate (C2);
\draw[line width=2] (C2) arc(-90:-135:8) -- ++(135:2) arc (45:90:8) coordinate 
(C3);
\draw[line width=2] (C3) -- ++ (-5,0) arc(270:90:6) coordinate (B1);

\draw[line width=6,color=white] (A2) arc(180:270:8);
\draw[line width=2] (A2) arc(180:270:8);
\draw[line width=6,color=white] (B1) arc(90:45:8) -- ++ (-45:2) arc (225:270:5) 
coordinate(B2);
\draw[line width=2] (B1) arc(90:45:8) -- ++ (-45:2) arc (225:270:5) 
coordinate(B2);
\draw[line width=6,color=white] (B2) arc(270:315:5) -- 
++ (60:11) arc (150:90:4.5);
\draw[line width=2] (B2) arc(270:315:5) -- 
++ (60:11) arc (150:90:4.5) coordinate(B3);
\draw[line width=2] (B3) arc(90:35:4.5) -- ++ (-55:2) arc 
(225:270:4) coordinate (B4);
\draw[line width=6,color=white] (B4) arc(270:330:4) -- ++ (60:16);
\draw[line width=2] (B4) arc(270:330:4) -- ++ (60:16) arc (135:90:5)  coordinate (IA);
\draw[line width=2] (-2,0) -- ++ (22,0) arc (-90:-45:8) -- ++ (45:2) 
arc(135:90:8) coordinate (D);
\draw[line width=2] (D) arc(90:45:8)  -- ++ (-45:2) arc (-135:-90:8) 
-- ++ (18,0);
\draw[line width=2] (C1) arc (90:135:8) -- ++ (225:2) arc (-45:-90:8);

\draw[line width=6,color=white] (D) arc(90:45:8)  -- ++ (-45:2);
\draw[line width=2] (D) arc(90:45:8)  -- ++ (-45:2);
\draw[line width=6,color=white] (C2) arc(-90:-135:8) -- ++(135:2) arc 
(45:90:8);
\draw[line width=2] (C2) arc(-90:-135:8) -- ++(135:2) arc 
(45:90:8);
\draw[line width=6,color=white] (A10) arc(-90:-35:4.5)-- ++ (55:2) 
arc(-225:-270:4) coordinate (A11);
\draw[line width=2,color=black] (A10) arc(-90:-35:4.5)-- ++ (55:2) 
arc(-225:-270:4) coordinate (A11);
\begin{scope}[shift={(60,0)}]
 \draw[line width=2](5,40) coordinate (JA) arc(90:45:5) 
coordinate (A0);
\draw[line width=2] (A0) -- ++ (-45:2) arc (45:0:6) coordinate (A1);
\draw[line width=2] (A1) arc(0:-45:6) arc (135:180:8) coordinate (A2);
\draw[line width=2] (A2) arc(180:270:8) coordinate(A3);
\draw[line width=2] (A3) arc(270:395:8) coordinate (A4);
\draw[line width=2] (A4)-- ++ (125:10) coordinate (A5);
\draw[line width=2] (A5) arc(45:180:5) coordinate (A6);
\draw[line width=2] (A6) arc(180:270:5) coordinate(A7);
\draw[line width=2] (A7) arc(270:315:5) --++ (45:8) coordinate(A8);
\draw[line width=2] (A4)-- ++ (125:10);
\draw[line width=2] (A8) arc (135:45:5) -- ++ (-60:5) coordinate(A9);
\draw[line width=2] (A9) -- ++ (-60:6) arc(-150:-90:4.5) coordinate (A10);
\draw[line width=2] (A10) arc(-90:-35:4.5)-- ++ (55:2) arc(-225:-270:4) 
coordinate (A11);
\draw[line width=2] (A11)  arc(90:-90:10) coordinate (C1);
\draw (C1) arc (90:135:8) -- ++ (225:2) arc (-45:-90:8) coordinate (C2);
\draw[line width=2] (C2) arc(-90:-135:8) -- ++(135:2) arc (45:90:8) coordinate 
(C3);
\draw[line width=2] (C3) -- ++ (-5,0) arc(270:90:6) coordinate (B1);
\draw[line width=2] (B1) arc(90:45:8) -- ++ (-45:2) arc (225:270:5) 
coordinate(B2);
\draw[line width=2] (B2) arc(270:315:5) -- 
++ (60:11) arc (150:90:4.5) coordinate(B3);
\draw[line width=2] (B3) arc(90:35:4.5) -- ++ (-55:2) arc 
(225:270:4) coordinate (B4);
\draw[line width=2] (B4) arc(270:330:4) -- ++ (60:16) arc (135:90:4) 
-- ++ (5,0);
\draw[line width=2] (-2,0) -- ++ (22,0) arc (-90:-45:8) -- ++ (45:2) 
arc(135:90:8) coordinate (D);
\draw[line width=2] (D) arc(90:45:8)  -- ++ (-45:2) arc (-135:-90:8) 
-- ++ (15,0);
\draw[line width=2] (C1) arc (90:135:8) -- ++ (225:2) arc (-45:-90:8);

\draw[line width=6,color=white] (A1) arc(0:-45:6) arc (135:180:8) coordinate (A2);
\draw[line width=2] (A1) arc(0:-45:6) arc (135:180:8) coordinate (A2);
\draw[line width=6,color=white] (A3) arc(270:395:8) coordinate (A4);
\draw[line width=2] (A3) arc(270:395:8) coordinate (A4);
\draw[line width=6,color=white] (A5) arc(45:180:5) coordinate (A6);
\draw[line width=2] (A5) arc(45:180:5) coordinate (A6);
\draw[line width=6,color=white] (A7) arc(270:315:5) --++ (45:8) coordinate(A8);
\draw[line width=2] (A7) arc(270:315:5) --++ (45:8) coordinate(A8);
\draw[line width=6,color=white] (A9) -- ++ (-60:6) arc(-150:-90:4.5) coordinate (A10);
\draw[line width=2] (A9) -- ++ (-60:6) arc(-150:-90:4.5) coordinate (A10);
\draw[line width=6,color=white] (A11)  arc(90:-90:10) coordinate (C1);
\draw[line width=2] (A11)  arc(90:-90:10) coordinate (C1);
\draw[line width=6,color=white] (C1) arc (90:135:8) -- ++ (225:2) arc (-45:-90:8) coordinate (C2);
\draw[line width=2] (C1) arc (90:135:8) -- ++ (225:2) arc (-45:-90:8) coordinate (C2);
\draw[line width=6,color=white] (C3) -- ++ (-5,0) arc(270:90:6) coordinate (B1);
\draw[line width=2] (C3) -- ++ (-5,0) arc(270:90:6) coordinate (B1);
\draw[line width=6,color=white] (B3) arc(90:35:4.5) -- ++ (-55:2) arc 
(225:270:4) coordinate (B4);
\draw[line width=2] (B3) arc(90:35:4.5) -- ++ (-55:2) arc 
(225:270:4) coordinate (B4);
\draw[line width=6,color=white] (-2,0) -- ++ (22,0) arc (-90:-45:8) -- ++ (45:2) 
arc(135:90:8) coordinate (D);
\draw[line width=2] (-2,0) -- ++ (22,0) arc (-90:-45:8) -- ++ (45:2) 
arc(135:90:8) coordinate (D);

\end{scope}
\draw[line width=2] (IA) -- (JA);

\end{tikzpicture}
\end{center}
\caption{The key tangle $M_1=T_{20}$.}
\label{F:M1}
\end{figure}
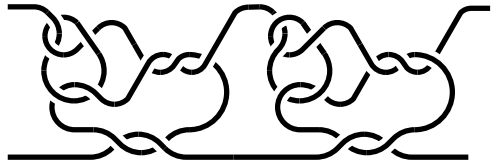

\begin{defi}
For $r\in\NN\setminus\{0\}$, we define $K_r$ to be the knot represented by the diagram~$\den(1\ast M_r)$.
\end{defi}

By definition, the tangle $M_r$ has $2^{r-1}\times 20$ crossings.
So the knot $K_r$ has at most $1+2^{r-1}\times 20$ crossings.
We observe that the tangle $M_{r}$ is the union of two arcs, 
the first one going from NW to NE and the second one from SW to SE.
As illustrated by the diagram
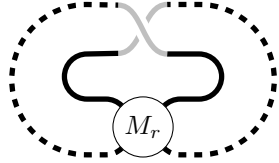
\begin{center}
 \begin{tikzpicture}[x=0.05cm,y=0.05cm]
 \draw (10,0) coordinate (O) circle (8) node {$M_r$};
 \draw (O) ++ (45:8) coordinate (NE);
 \draw (O) ++ (135:8) coordinate (NW);
 \draw (O) ++ (225:8) coordinate (SW);
 \draw (O) ++ (315:8) coordinate (SE);
 \draw[line width=2] (NE) arc (135:90:6) -- ++ (5,0) arc (-90:90:6) --
++ (-8.5,0) coordinate (B1);
 \draw[line width=2,dashed] (SE) arc (225:270:6) -- ++ (5,0) arc (-90:90:20) 
-- ++ (-8.5,0) coordinate (A1);
\draw (A1) ++ (-6.5,0) coordinate (A2);
\draw (A1) ++ (-6.5,-13) coordinate (A3);
\draw (A1) ++ (-13,-13) coordinate (A4);
\draw[line width=2,color=lightgray] (A1) .. controls (A2) and (A3) .. (A4);
\draw (B1) ++ (-6.5,0) coordinate (B2);
\draw (B1) ++ (-6.5,13) coordinate (B3);
\draw (B1) ++ (-13,13) coordinate (B4);
\draw[line width=6,color=white] (B1) .. controls (B2) and (B3) .. (B4);
\draw[line width=2,color=lightgray] (B1) .. controls (B2) and (B3) .. (B4);
 \draw[line width=2] (NW) arc (45:90:6) -- ++ (-5,0) arc (270:90:6)  -- (A4);
 \draw[line width=2,dashed] (SW) arc (-45:-90:6) -- ++ (-5,0) arc (270:90:20) 
-- (B4);
 \end{tikzpicture}
\end{center}
the link diagram $\den(1\ast M_r)$ has exactly one component: if we travel along the dotted 
arc we must meet the undotted one. 
The following proposition summarizes these remarks.

\begin{prop}
For each  $r\geq 1$, the link $K_r$ is  a knot with at most
$1+2^{r-1}\times 20$ crossings.
\end{prop}

For all $r\geq 1$, a mutant knot $K_r'$ is obtained from $K_r$ by replacing the tangle $-T_{10}$ in each summand $M_1$ of $M_r$ by its image under vertical symmetry.
The knot~$K_r$ and~$K_r'$, being mutant of each other, have the same Jones polynomial~\cite{MR975096}.
The knot~$K'_1$ is depicted on Figure~\ref{F:Knot}.

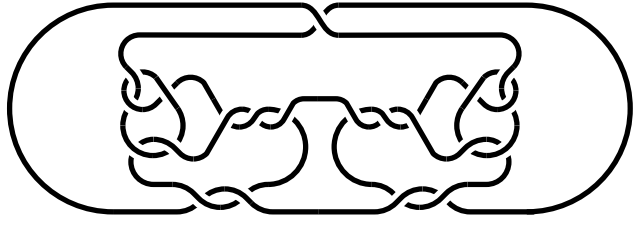
\begin{figure}
\begin{center}
 \begin{tikzpicture}[x=0.05cm,y=0.05cm]
\draw[line width=2,color=white] (5,40) arc(90:45:5) coordinate (A0);
\draw[line width=2] (A0) -- ++ (-45:2) arc (45:0:6) coordinate (A1);
\draw[line width=2] (A1) arc(0:-45:6) arc (135:180:8) coordinate (A2);
\draw[line width=2] (A2) arc(180:270:8) coordinate(A3);
\draw[line width=2] (A3) arc(270:395:8) coordinate (A4);
\draw[line width=2] (A4)-- ++ (125:10) coordinate (A5);
\draw[line width=2] (A5) arc(45:180:5) coordinate (A6);
\draw[line width=6,color=white] (A0) -- ++ (-45:2) arc (45:0:6);
\draw[line width=2] (A0) -- ++ (-45:2) arc (45:0:6);
\draw[line width=6,color=white] (A6) arc(180:270:5) coordinate(A7);
\draw[line width=2] (A6) arc(180:270:5) coordinate(A7);
\draw[line width=2] (A7) arc(270:315:5) --++ (45:8) coordinate(A8);
\draw[line width=6,color=white] (A4)-- ++ (125:10);
\draw[line width=2] (A4)-- ++ (125:10);
\draw[line width=2] (A8) arc (135:45:5) -- ++ (-60:5) coordinate(A9);
\draw[line width=2] (A9) -- ++ (-60:6) arc(-150:-90:4.5) coordinate (A10);
\draw[line width=2] (A10) arc(-90:-35:4.5)-- ++ (55:2) arc(-225:-270:4) coordinate (A11);
\draw[line width=2] (A11)  arc(90:-90:10) coordinate (C1);
\draw (C1) arc (90:135:8) -- ++ (225:2) arc (-45:-90:8) coordinate (C2);
\draw[line width=2] (C2) arc(-90:-135:8) -- ++(135:2) arc (45:90:8) coordinate (C3);
\draw[line width=2] (C3) -- ++ (-5,0) arc(270:90:6) coordinate (B1);
\draw[line width=6,color=white] (A2) arc(180:270:8);
\draw[line width=2] (A2) arc(180:270:8);
\draw[line width=6,color=white] (B1) arc(90:45:8) -- ++ (-45:2) arc (225:270:5) coordinate(B2);
\draw[line width=2] (B1) arc(90:45:8) -- ++ (-45:2) arc (225:270:5) coordinate(B2);
\draw[line width=6,color=white] (B2) arc(270:315:5) -- ++ (60:11) arc (150:90:4.5);
\draw[line width=2] (B2) arc(270:315:5) -- ++ (60:11) arc (150:90:4.5) coordinate(B3);
\draw[line width=2] (B3) arc(90:35:4.5) -- ++ (-55:2) arc (225:270:4) coordinate (B4);
\draw[line width=6,color=white] (B4) arc(270:330:4) -- ++ (60:16);
\draw[line width=2] (B4) arc(270:330:4) -- ++ (60:5) arc (135:90:4) -- ++ (5,0);
\draw[line width=2] (5,0) coordinate (SW) -- ++ (17,0) arc (-90:-45:8) -- ++ (45:2) arc(135:90:8) coordinate (D);
\draw[line width=2] (D) arc(90:45:8)  -- ++ (-45:2) arc (-135:-90:8) -- ++ (15,0);
\draw[line width=2] (C1) arc (90:135:8) -- ++ (225:2) arc (-45:-90:8);
\draw[line width=6,color=white] (D) arc(90:45:8)  -- ++ (-45:2);
\draw[line width=2] (D) arc(90:45:8)  -- ++ (-45:2);
\draw[line width=6,color=white] (C2) arc(-90:-135:8) -- ++(135:2) arc (45:90:8);
\draw[line width=2] (C2) arc(-90:-135:8) -- ++(135:2) arc (45:90:8);
\draw[line width=6,color=white] (A10) arc(-90:-35:4.5)-- ++ (55:2) arc(-225:-270:4) coordinate (A11);
\draw[line width=2] (A10) arc(-90:-35:4.5)-- ++ (55:2) arc(-225:-270:4) coordinate (A11);
\draw[line width=2] (A0) arc(225:90:5) -- (55,47) coordinate (SW1);
\begin{scope}[shift={(120,0)},xscale=-1]
\draw[line width=2,color=white] (5,40) arc(90:45:5) coordinate (A0);
\draw[line width=2] (A0) -- ++ (-45:2) arc (45:0:6) coordinate (A1);
\draw[line width=2] (A1) arc(0:-45:6) arc (135:180:8) coordinate (A2);
\draw[line width=2] (A2) arc(180:270:8) coordinate(A3);
\draw[line width=2] (A3) arc(270:395:8) coordinate (A4);
\draw[line width=2] (A4)-- ++ (125:10) coordinate (A5);
\draw[line width=2] (A5) arc(45:180:5) coordinate (A6);
\draw[line width=2] (A6) arc(180:270:5) coordinate(A7);
\draw[line width=2] (A7) arc(270:315:5) --++ (45:8) coordinate(A8);
\draw[line width=2] (A8) arc (135:45:5) -- ++ (-60:5) coordinate(A9);
\draw[line width=2] (A9) -- ++ (-60:6) arc(-150:-90:4.5) coordinate (A10);
\draw[line width=2] (A10) arc(-90:-35:4.5)-- ++ (55:2) arc(-225:-270:4) coordinate (A11);
\draw[line width=2] (A11)  arc(90:-90:10) coordinate (C1);
\draw[line width=2] (C1) arc (90:135:8) -- ++ (225:2) arc (-45:-90:8) coordinate (C2);
\draw[line width=2] (C2) arc(-90:-135:8) -- ++(135:2) arc (45:90:8) coordinate (C3);
\draw[line width=2] (C3) -- ++ (-5,0) arc(270:90:6) coordinate (B1);
%\draw[line width=2] (A2) arc(180:270:8);
\draw[line width=2] (B1) arc(90:45:8) -- ++ (-45:2) arc (225:270:5) coordinate(B2);
\draw[line width=2] (B2) arc(270:315:5) -- ++ (60:11) arc (150:90:4.5) coordinate(B3);
\draw[line width=2] (B3) arc(90:35:4.5) -- ++ (-55:2) arc (225:270:4) coordinate (B4);
\draw[line width=2] (B4) arc(270:330:4) -- ++ (60:5) arc (135:90:4) -- ++ (5,0);
\draw[line width=2] (3,0) -- ++ (17,0) arc (-90:-45:8) -- ++ (45:2) arc(135:90:8) coordinate (D);
\draw[line width=2] (D) arc(90:45:8)  -- ++ (-45:2) arc (-135:-90:8) -- ++ (15,0);
%Specify crossings
\draw[line width=6,color=white] (A0) -- ++ (-45:2) arc (45:0:6) coordinate (A1);
\draw[line width=2] (A0) -- ++ (-45:2) arc (45:0:6) coordinate (A1);
\draw[line width=6,color=white] (A2) arc(180:270:8) coordinate(A3);
\draw[line width=2] (A2) arc(180:270:8) coordinate(A3);
\draw[line width=6,color=white] (A4)-- ++ (125:10) coordinate (A5);
\draw[line width=2] (A4)-- ++ (125:10) coordinate (A5);
\draw[line width=6,color=white] (A6) arc(180:270:5) coordinate(A7);
\draw[line width=2] (A6) arc(180:270:5) coordinate(A7);
\draw[line width=6,color=white] (A10) arc(-90:-35:4.5)-- ++ (55:2) arc(-225:-270:4) coordinate (A11);
\draw[line width=2] (A10) arc(-90:-35:4.5)-- ++ (55:2) arc(-225:-270:4) coordinate (A11);
\draw[line width=6,color=white] (C2) arc(-90:-135:8) -- ++(135:2) arc (45:90:8) coordinate (C3);
\draw[line width=2] (C2) arc(-90:-135:8) -- ++(135:2) arc (45:90:8) coordinate (C3);
\draw[line width=6,color=white] (B1) arc(90:45:8) -- ++ (-45:2) arc (225:270:5) coordinate(B2);
\draw[line width=2] (B1) arc(90:45:8) -- ++ (-45:2) arc (225:270:5) coordinate(B2);
\draw[line width=6,color=white] (B2) arc(270:315:5) -- ++ (60:11) arc (150:90:4.5) coordinate(B3);
\draw[line width=2] (B2) arc(270:315:5) -- ++ (60:11) arc (150:90:4.5) coordinate(B3);
\draw[line width=6,color=white] (B4) arc(270:330:4) -- ++ (60:5) arc (135:90:4) -- ++ (5,0);
\draw[line width=2] (B4) arc(270:330:4) -- ++ (60:5) arc (135:90:4) -- ++ (5,0);
\draw[line width=6,color=white] (D) arc(90:45:8)  -- ++ (-45:2) arc (-135:-90:8) -- ++ (15,0);
\draw[line width=2] (D) arc(90:45:8)  -- ++ (-45:2) arc (-135:-90:8) -- ++ (15,0);
\draw[line width=2] (A0) arc(225:90:5) -- (55,47) coordinate (SE1);
\end{scope}
\draw[line width=2](SW) arc (270:90:27.5) -- (55,55) coordinate (NW1);
\draw[line width=2](115,0) coordinate (SE) arc (-90:90:27.5) -- (65,55) coordinate (NE1);
\draw[line width=2] (SW1) .. controls (60,47) and (60,55) .. (NE1);
\draw[line width=6,color=white] (NW1) .. controls (60,55) and (60,47) .. (SE1);
\draw[line width=2] (NW1) .. controls (60,55) and (60,47) .. (SE1);
\end{tikzpicture}
\end{center}
\caption{
An elegant crab-like mutant version $K'_1$ of the knot $K_1$.
%A mutant version $K'_1$ of the knot $K_1$ obtained by replacing the tangle $-T_{10}$ in $M_1$ by its image under vertical symmetry.
}
\label{F:Knot}
\end{figure}

\section{The Kauffman bracket pair of a tangle}
\label{S:bracket}

In this section we recall the definition of the Kauffman bracket pair of a tangle.
This notion is very powerful to determine the Jones polynomial for a knot 
obtained from an algebraic tangle, which is the case for our knots $K_r$.

Let $T$ be a tangle with $4$ endpoints. 
The Kauffman bracket $\kbr{T}$ of $T$ is a linear combination of two 
formal symbols $\kbr{0}$ and $\kbr{\infty}$ with coefficients in the ring of 
Laurent polynomials $\Lambda=\ZZ[t,t\inv]$.
The bracket $\kbr{T}$ may be computed with the usual rules of the Kauffman
bracket polynomial, namely:

-- $\kbr{\unknot}=1$;

-- $\kbr{\unknot\amalg T}=\delta\kbr{T}$ where $\delta=-t^{-2}-t^2$;

-- $\kbr{\toneinv}=t\inv\kbr{\tz}+t\kbr{\ti}.$

Thus, after removing the crossings and all free loops using the rules 
above, we end up with a unique expression of the form 
$\kbr{T}=f(T)\kbr{0}+g(T)\kbr{\infty}$ where $0$ is the 
tangle $\tz$ and $\infty$ is the tangle $\ti$.
The notation $0$ and $\infty$ comes from the shape of the link obtained when we 
take the $\den$ closure of the corresponding tangles. 

We define the \emph{bracket pair} $\br(T)$ of $T$ as 
\begin{equation}\label{E:Br}
\br(T)=\begin{bmatrix}f(T)\\g(T)\end{bmatrix}\in\Lambda^2.
\end{equation}
For example we compute
 $$
 \kbr{-1}=\kbr{\toneinv}=t\inv\kbr{\tz}+t\kbr{\ti}=t\inv\kbr{0}+t\kbr{\infty},
 $$ 
which implies 
$\br(-1)=\left[\begin{matrix}t\inv\\t\end{matrix}\right]$.
Since by definition of $-T$, we have $f(-T)=f(T)|_{t\leftarrow t\inv}$ and 
$g(-T)=g(T)|_{t\leftarrow t\inv}$ we obtain
$\br(1)=\left[\begin{matrix}t\\t\inv\end{matrix}\right]$. 

Proposition 2.2 of \cite{EKP} gives computation rules for the bracket pair of the
horizontal and vertical sum of tangles together with the bracket of the $\num$ 
and $\den$ closures of tangles. Here they are.

\begin{prop}
\label{P:Calc}
For two tangles $T$ and  $U$, we have:
\vspace{0.3em}

-- $\phantom{i}i)$ $\br(T\!+\!U)
=\left[\begin{matrix} f(T)f(U)\\ f(T) g(U) + g(T) f(U) + \delta g(T) g(U) \end{matrix}\right]$;
\vspace{0.5em}

-- $ii)$ $\br(T\ast U)=\left[\begin{matrix}\delta f(T) f(U) + f(T) g(U) + g(T) f(U) \\ g(T) g(U)\end{matrix}\right]$;
\vspace{0.5em}

-- $iii)$ $\kbr{\num(T)}=\delta f(T)+g(T)$ and $\kbr{\den(T)}=f(T)+\delta g(T)$.

\end{prop}

A direct computation from the expression of $\br(1)$ gives
$$
\br(2)=\begin{bmatrix}
        t^2\\
        -t^{-4}+1
       \end{bmatrix},\quad
\br(3)=\begin{bmatrix}
        t^3\\
        t^{-7}-t^{-3}+t
       \end{bmatrix}\quad\text{and}\quad
\br(1/2)=\begin{bmatrix}
          1-t^4\\
          t^{-2}
         \end{bmatrix}.
$$
Using these values, we determine the bracket pair of 
$T_{8,21}=(((1/2)+1)\ast2)+(-3)$:
\begin{equation}
\label{E:T821}
\br(T_{8,21})=\begin{bmatrix}
                                          -2\,t^{-6}+2\,t^{-2}-2\,t^2+t^6\\
                                          
-2\,t^{-4}+3\,-4\,t^4+3\,t^8-2\,t^{12}+t^{16}
                                         \end{bmatrix}.
\end{equation}

\begin{nota}
In the sequel, for a tangle $T$ and an integer $m\geq 2$, we will denote 
by $\br_m(T)$ the bracket pair of $T$ modulo $m$. 
\end{nota}

\smallskip

\subsection{The case of tangle $T_{20}$.}

We now analyze the bracket pair of the tangle $T_{20}$ introduced in Definition~\ref{D:Tangles}.

\begin{lem}
\label{L:ModBracket}
The  bracket pair  $\br_2(T_{20})$ is equal to $\left[\begin{smallmatrix}1\\0\end{smallmatrix} 
\right]$. 
Moreover, the leading term of $f(T_{20})$ is $2\,t^{28}$ and that of $g(T_{20})$ is $2\,t^{26}$.
\end{lem}

\begin{proof}
By relation~\eqref{E:T821}, we obtain
$$
\br(T_{10})=\br(T_{8,21}\ast 2)=
\begin{bmatrix}
2\,t^{-10} - 2\,t^{-6} + 2\,t^{-2} - 2\,t^6 + 2\,t^{10} - 2\,t^{14} + 
t^{18}\\
2\,t^{-8} - 5\,t^{-4} + 7 - 7\,t^4 + 5\,t^8 - 3\,t^{12} + t^{16}           
\end{bmatrix}.
$$
Replacing $t$ by $t\inv$, we get
$$
\br(-T_{10})=\begin{bmatrix}
t^{-18} - 2\,t^{-14} + 2\,t^{-10} - 2\,t^{-6} + 2\,t^2 - 2\,t^6 + 
2\,t^{10}\\
t^{-16} - 3\,t^{-12} + 5\,t^{-8} - 7\,t^{-4} + 7 - 5\,t^4 + 2\,t^8
\end{bmatrix}.
$$
The formula for the computation of $\br(T_{10}+(-T_{10}))$ given in $i)$ of
Proposition~\ref{P:Calc} implies that the leading term of $f(T_{20})$ is 
$t^{18}\cdot 2\,t^{10}=2\,t^{28}$ and that of $g(T_{20})$ is
$$t^{18}\cdot 2\,t^{8}+t^{16}\cdot 2\,t^{10}-t^2\cdot t^{16}\cdot 2\,t^8=2\,t^{26}$$
as expected.
Taking coefficient modulo $2$, we obtain
\begin{align*}
\br_2(T_{10})&=
\begin{bmatrix}
t^{18}\\
t^{-4} + 1 + t^4 + t^8 + t^{12} + t^{16}           
\end{bmatrix},
\\
\br_2(-T_{10})&=\begin{bmatrix}
t^{-18}\\
t^{-16} + t^{-12} + t^{-8} + t^{-4} + 1 +t^4
\end{bmatrix},
\end{align*}
and so $\br_2(T_{20})=\br_2(T_{10}+(-T_{10}))=\left[
\begin{smallmatrix}1\\0\end{smallmatrix}\right]$.
\end{proof}

\subsection{The general case of tangles $M_r$.}

We now analyze the bracket pair \eqref{E:Br} of the tangle $M_r$ constructed in Definition~\ref{D:Tangles}.
For convenience, for $r\geq 1$, we denote by $\ell_r\in\ZZ[t,t\inv]$ the leading term of $f(M_r)$.

\begin{prop}
\label{P:M20}
For all $r\geq 1$, we have $\br_{2^r}(M_r)=\left[\begin{smallmatrix}1\\0\end{smallmatrix} 
\right]$. 
Moreover, we have $\ell_r=\left(2\,t^{28}\right)^{2^{r-1}}$ while the leading term  
of $g(M_r)$ is equal to $t^{-2}\ell_r$.
\end{prop}

\begin{proof}
By induction on $r\geq 1$. The case $r=1$ is Lemma~\ref{L:ModBracket}.
Assume now $r\geq 2$.
By the induction hypothesis, we have $f(M_{r-1})\equiv 1 \mod 2^{r-1}$ and 
$g(M_{r-1})\equiv 0 \mod 2^{r-1}$.
Hence, there exist two Laurent polynomials $P$ and $Q$ in $\ZZ[t,t\inv]$ such that the 
relations $f(M_{r-1})=1+2^{r-1}P$ and $g(M_{r-1})=2^{r-1}Q$ hold.
By Proposition~\ref{P:Calc} and formula $\br(M_r)=\br(M_{r-1}+M_{r-1})$ we obtain:
\begin{align*}
\br(M_r)&=\begin{bmatrix}
   f(M_{r-1})^2\\
   2g(M_{r-1})\,f(M_{r-1})+\delta g(M_{r-1})^2
  \end{bmatrix}
  \\
&=\begin{bmatrix}
    1+2^r P+ 2^{2r-2} P^2\\
   2^r Q+2^{2r-1}PQ+\delta 2^{2r-2}Q^2
   \end{bmatrix}.
\end{align*}
As $2r-2\geq r$ holds since $r\geq 2$, we find
$\br_{2^r}(M_r)=\left[\begin{smallmatrix}1\\0\end{smallmatrix} 
\right]$.
Let us now prove the statement about the leading terms.
Since $f(M_r)$ is equal to $f(M_{r-1})^2$, the induction hypothesis
implies that the leading term of $f(M_r)$ is the square of the leading term of 
$f(M_{r-1})$, \ie, the square of $\ell_{r-1}$.
Since $\ell_{r-1}^2$ is equal to $\ell_r$, we have the desired result for the 
leading term of $f(M_r)$.
Denoting by $\lt(P)$ the leading term of a Laurent polynomial $P\in\Lambda$, we have
$$\lt(g(M_r))=\lt\left(2g(M_{r-1})f(M_{r-1})+\delta\,g(M_{r-1})^2\right).$$
By the induction hypothesis, we compute
\begin{align*}
 \lt(g(M_{r-1})\,f(M_{r-1}))&=t^{-2}\ell_{r-1}\cdot\ell_{r-1}=t^{-2}\ell_{r-1}^2=t^{-2}\ell_r\\
 \lt(\delta\,g(M_{r-1}^2))&=-t^{2}\cdot(t^{-2}\ell_{r-1})^2=-t^{-2}\ell_{r-1}^2=-t^{-2}\ell_r
\end{align*}
and so $\lt(g(M_r))=2\left(t^{-2}\ell_r\right)-\left(t^{-2}\ell_r\right)=t^{-2}\ell_r$, as desired.
\end{proof}

\section{On the Jones polynomial of $K_r$}
\label{S:jones}

The \emph{writhe} of an oriented link diagram $D$, denoted by $\wri(D)$, is the sum of the 
signs of the crossings of $D$ following conventions \textbf{a} and 
\textbf{b} of Figure~\ref{F2}.

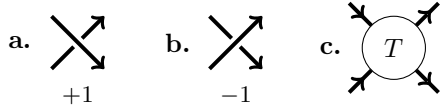
\begin{figure}[h!]

\begin{center}
 \begin{tikzpicture}[x=0.07cm,y=0.07cm]
 \draw(-6,5) node {\textbf{a.}};
  \draw[line width=1.5,->](0,0) -- (10,10);
  \draw[line width=4.5,color=white](0,10) -- (10,0);
  \draw[line width=1.5,->](0,10) -- (10,0);
  \draw(5,-5) node{\small$+1$};
  \begin{scope}[shift={(30,0)}]
  \draw(-6,5) node{\textbf{b.}};
    \draw[line width=1.5,->](0,10) -- (10,0);
  \draw[line width=4.5,color=white](0,0) -- (10,10);
  \draw[line width=1.5,->](0,0) -- (10,10);
  \draw(5,-5) node{\small$-1$};
  \end{scope}
  \begin{scope}[shift={(65,4)}]
   \draw(-12,0) node{\textbf{c.}};
   \draw(0,0)  circle (6) node{$T$};
   \draw[line width=1.5,-<] (45:12) -- (45:8);
   \draw[line width=1.5](45:9) -- (45:6);
   \draw[line width=1.5,->] (135:12) -- (135:8);
   \draw[line width=1.5](135:9) -- (135:6);
   \draw[line width=1.5,->] (225:12) -- (225:8);
   \draw[line width=1.5](225:9) -- (225:6);
   \draw[line width=1.5,-<] (315:12) -- (315:8);
   \draw[line width=1.5](315:9) -- (315:6);
  \end{scope}
 \end{tikzpicture}
\end{center}
\caption{\textbf{a.} Crossing with positive writhe. \textbf{b.} Crossing with 
negative writhe. \textbf{c.} Left-right oriented tangle.}
\label{F2}
\end{figure}

The notion of writhe can be naturally extended  to an oriented tangle $T$.
We say that a tangle $T$ is \emph{left-right orientable} if it can be equipped with 
an orientation as in~\textbf{c} of Figure~\ref{F2}, in which case we denote by $\wri(T)$ 
 the writhe of $T$ with respect to that orientation.

\begin{lem}
\label{L:Writhe}
The tangle $M_r$ is left-right orientable and we have $\wri(M_r)=0$.  
\end{lem}

\begin{proof}
We first remark that the tangle $T_{10}$ is left-right orientable.
 Since $-T_{10}$ is obtained from $T_{10}$ by switching the signs of the 
crossings, the tangle $-T_{10}$ is also left-right orientable and 
$\wri(-T_{10})=-\wri(T_{10})$ holds.
 As the horizontal sum of tangles is compatible with the left-right 
orientation, for any left-right orientable tangles $U$ and~$V$, we have 
$\wri(U+V)=\wri(U)+\wri(V)$.
As $M_1=T_{10}+(-T_{10})$, we have 
$$\wri(M_1)=\wri(T_{10})+\wri(-T_{10})=\wri(T_{10})-\wri(T_{10})=0.$$
Again by the compatibility between the left-right orientation and the 
horizontal sum $+$, a straightforward induction yields
$\wri(M_r)=\wri(M_{r-1})+\wri(M_{r-1})=0+0=0$.
\end{proof}

The \emph{normalized Kauffman bracket polynomial} of a link $L$ depicted by an oriented
diagram~$D$ is $$\chi(L)=(-t^3)^{-\wri(D)}\kbr{D},$$ which is an invariant of 
the link $L$ and so is independent of the choice of the oriented diagram $D$ representing $L$.
The \emph{Jones polynomial} of a link~$L$ is then $$V(L)=\chi(L)|_{t\leftarrow
t^{-1/4}}.$$
The Kauffman bracket of the unknot $\unknot$ is $\kbr{\den(0)}=1$, which gives 
$\chi(\unknot)=1$ and so $V(\unknot)=1$.

Recall that $K_r$ is the knot represented by the diagram $\den(1\ast M_r)$ and that $\ell_r$ is the 
leading term of $f(M_r)$.

\begin{prop}
\label{P:Triv}
 For all $r\geq 1$, the Jones polynomial of $K_r$ is equal to 
$1$ modulo~$2^r$. Moreover the leading term of $\chi(K_r)$ is equal to $\ell_r$.
\end{prop}

\begin{proof}
Let $r$ be an integer $\geq 1$.  We denote by $D_r$ the diagram $\den(1\ast M_r)$.
The left-right orientation of $M_r$ induces the following orientation on~$D_r$:

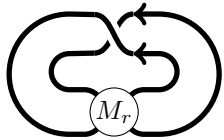
\begin{center}
  \begin{tikzpicture}[x=0.04cm,y=0.04cm]
 \draw (10,0) coordinate (O) circle (8) node {$M_r$};
 \draw (O) ++ (45:8) coordinate (NE);
 \draw (O) ++ (135:8) coordinate (NW);
 \draw (O) ++ (225:8) coordinate (SW);
 \draw (O) ++ (315:8) coordinate (SE);
 \draw[line width=2,->] (NE) arc (135:90:6) -- ++ (5,0) arc (-90:90:6) --
++ (-8.5,0) coordinate (B1);
 \draw[line width=2,->] (SE) arc (225:270:6) -- ++ (5,0) arc (-90:90:20) 
-- ++ (-8.5,0) coordinate (A1);
\draw (A1) ++ (-6.5,0) coordinate (A2);
\draw (A1) ++ (-6.5,-13) coordinate (A3);
\draw (A1) ++ (-13,-13) coordinate (A4);
\draw[line width=2] (A1) .. controls (A2) and (A3) .. (A4);
\draw (B1) ++ (-6.5,0) coordinate (B2);
\draw (B1) ++ (-6.5,13) coordinate (B3);
\draw (B1) ++ (-13,13) coordinate (B4);
\draw[line width=6,color=white] (B1) .. controls (B2) and (B3) .. (B4);
\draw[line width=2] (B1) .. controls (B2) and (B3) .. (B4);
 \draw[line width=2] (NW) arc (45:90:6) -- ++ (-5,0) arc (270:90:6)  -- (A4);
 \draw[line width=2] (SW) arc (-45:-90:6) -- ++ (-5,0) arc (270:90:20) 
-- (B4);
 \end{tikzpicture}
\end{center}
As the writhe of $M_r$ is $0$ by Lemma~\ref{L:Writhe}, the writhe of $D_r$ with respect to the above orientation is 
$+1$.

Let us now determine the Kauffman bracket of $D_r$.
We have 
$$
\br(1\ast M_r)=\begin{bmatrix}
  \delta\,t\,f(M_r)+t\,g(M_r)+t\inv\,f(M_r)\\
 t\inv\,g(M_r)
  \end{bmatrix}
$$
and so 
\begin{align*}
\kbr{D_r}&= 
\delta\,t\,f(M_r)+t\,g(M_r)+t\inv\,f(M_r)+\delta\,t\inv\,g(M_r)\\
&=(\delta\,t+t\inv)\, f(M_r)+(t+\delta\,t\inv)\, g(M_r)\\
&=-t^3\,f(M_r)-t^{-3}\,g(M_r).
\end{align*}
Hence the normalized Kauffman bracket of $K_r$ is 
\begin{align*}
\chi(K_r)&=(-t^3)^{-\wri(D_r)}\cdot\kbr{D_r}=(-t^3)\inv\,(-t^3\,f(M_r)-t^{-3}\,
g(M_r))\\
&=f(M_r)+t^{-6}g(M_r).
\end{align*}
As $f(M_r)=1$ and $g(M_r)=0$ modulo 
$2^r$ by Proposition~\ref{P:M20}, we obtain $\chi(K_r)=1$ modulo $2^r$ and so $V(K_r)$ is trivial modulo 
$2^r$.
Since the leading term of $g(M_r)$ is equal to $t^{-2}\ell_r$ by Proposition~\ref{P:M20}, the leading term of $\chi(K_r)$ is $\ell_r$.
\end{proof}

We can now state and prove our main result.

\begin{thrm}
 For all $r\geq 1$, there exist infinitely many pairwise distinct knots with trivial Jones polynomial
 modulo $2^r$.
\end{thrm}

\begin{proof}
Let $r\geq 1$ be an integer.
The knots $K_i$ with $i\geq r$ satisfy the statement.
Indeed, by Proposition~\ref{P:Triv}, for all $i\geq r$ the Jones polynomial of $K_i$ is trivial
modulo~$2^i$ and thus modulo $2^r$.
Since for $j\geq 1$, the leading term of~$\chi(K_j)$ is 
$$\ell_j=\left(2\,t^{28}\right)^{2^{j-1}}$$ 
by Proposition~\ref{P:M20}, the map
$j\mapsto \chi(K_j)$ is injective.
In particular the knots $K_i$ for $i\geq r$ have distinct Jones polynomials and so they are pairwise distinct.
\end{proof}

\section{Concluding remarks}
We have exhibited a 20-crossing tangle $T_{20}$ whose Kauffman bracket polynomial pair is trivial mod $2$, i.e. congruent to $\left[\begin{smallmatrix}1\\0\end{smallmatrix} 
\right]$ mod $2$. That tangle allowed us to construct, for any $r \ge 1$, a nontrivial knot $K_r$ whose Jones polynomial is congruent to 1 mod $2^r$.

\smallskip

Having thus solved Problem~\ref{Prob1} for $m=2^r$ and knowing solutions for $m=3$ from the Tables,
what about the existence of solutions for the next moduli, such as $m=5$, $6$ or $7$ for instance? 
More ambitiously perhaps, given an integer $m \ge 3$, does there exist a tangle, analogous to $T_{20}$,
whose Kauffman bracket polynomial pair would be trivial mod $m$? 
If yes, what should be the expected minimal number of crossings as a function of $m$?

\smallskip
Even more intriguing: does there exist a tangle whose Kauffman bracket polynomial pair is trivial \emph{over $\mathbb{Z}$}? The existence of such a tangle would probably imply the existence of a nontrivial knot with trivial Jones polynomial.

\smallskip
\bigskip
\noindent
\textbf{Acknowledgments.} The authors wish to thank L.~H.~Kauffman, D.~Rolfsen and M.~B.~Thistlethwaite for 
interesting discussions related to this work. They also wish to thank the anonymous referee for his very sharp and constructive comments.

\bibliographystyle{plain}
\bibliography{biblio}
\bigskip

{\small
\noindent\textbf{Authors addresses:}

\noindent
Shalom Eliahou, Jean Fromentin\textsuperscript{a,b}

\noindent
\textsuperscript{a}Univ. Littoral C\^ote d'Opale, EA 2597 - LMPA - Laboratoire de Math\'ematiques Pures et Appliqu\'ees Joseph Liouville, F-62228 Calais, France\\
\textsuperscript{b}CNRS, FR 2956, France\\
e-mail: \tt{eliahou@lmpa.univ-littoral.fr, fromentin@math.cnrs.fr}
}
\end{document}